\documentclass[a4paper, 11pt]{amsart} 
\usepackage{amssymb, latexsym, amsrefs, amsmath, amsthm, color, 
cancel}
\usepackage{a4wide}
\usepackage[mathscr]{euscript} \DeclareMathAlphabet{\euls}{U}{eus}{m}{n}

\numberwithin{equation}{section}
 \theoremstyle{plain}

\newtheorem{theorem}[equation]{Theorem}
\newtheorem{lemma}[equation]{Lemma}
\newtheorem{sublemma}[equation]{Sublemma}
\newtheorem{proposition}[equation]{Proposition}

\newtheorem{corollary}[equation]{Corollary}

\theoremstyle{definition}

\newtheorem{definition}[equation]{Definition}
\newtheorem{remark}[equation]{Remark}
\newtheorem{remarks}[equation]{Remarks}
\newtheorem{ex}[equation]{Example}
\newtheorem{conjecture}[equation]{Conjecture}
\newtheorem{question}[equation]{Question}

\newcommand{\udim}{\operatorname{udim}}
\newcommand{\socle}{\operatorname{socle}}
\newcommand{\Kdim}{\operatorname{Kdim}}
\newcommand{\cchar}{\operatorname{char}}
\newcommand{\Cent}{\operatorname{Cent}}
\newcommand{\GKdim}{\operatorname{GKdim}}

\begin{document}
\title[ the Dixmier-Moeglin Equivalence]{Pointed Hopf algebras, the Dixmier-Moeglin Equivalence and noetherian Group Algebras}

\author[J. P. Bell]{Jason P. Bell} 
\address{ Department of Pure Mathematics\\
University of Waterloo\\
Waterloo, ON CANADA N2L 3G1}
\email{jpbell@uwaterloo.ca}

\author[K. A. Brown]{Ken A. Brown}
\address{School of Mathematics and Statistics\\
University of Glasgow\\ Glasgow G12 8QW\\
Scotland}
\email{ken.brown@glasgow.ac.uk}

\thanks{The work of K.A.B. was supported by Leverhulme Emeritus Fellowship EM 2017-081.\\
 \indent  The work of J.T.S was supported by Leverhulme Emeritus Fellowship EM-2019-015.\\
 \indent The work of J.P.B was supported by NSERC grant RGPIN RGPIN-2022-02951}

\author[J.\ T.\ Stafford]{J. Toby Stafford}
\address{Department of Mathematics, The University of Manchester, Manchester M13 9PL, England}
\email{Toby.Stafford@manchester.ac.uk}


\date{\today}
 
\begin{abstract}This paper addresses the interactions between three properties that a group algebra or more generally a pointed Hopf algebra 
may possess:
 being noetherian, having finite Gelfand-Kirillov dimension, and satisfying the Dixmier-Moeglin equivalence. First it is shown that the
 second and third of these properties 
are equivalent for group algebras $kG$ of polycyclic--by-finite groups, and are, in turn, equivalent to $G$ being nilpotent-by-finite. 
In characteristic $0$, this enables us to extend this equivalence to certain cocommutative Hopf 
algebras. 

 In sections 3 and 4 of the paper finiteness conditions for group algebras are studied. Thus in $\S$3 we examine when a group algebra satisfies the Goldie conditions, while in the final section we discuss what can be said about a minimal counterexample to the conjecture that if $kG$ is noetherian then $G$ is polycyclic-by-finite. 
 \end{abstract}
 \maketitle

\section{Introduction}\label{intro} 
This paper concerns three separate but connected topics.

First, in $\S$\ref{group} we explore two aspects of the following conjecture which was proposed (for the case $k = \mathbb{C}$) as 
\cite[Conjecture~5.5]{BS}. 
For details of the terminology and notation, see $\S$\ref{notation} and Definition~\ref{DMEdefn}.

\begin{conjecture}\label{wild} Let $k$ be an algebraically closed field of characteristic zero and let $H$ be a noetherian pointed Hopf $k$-algebra. Then the following are equivalent:
\begin{enumerate}
\item[(1)] The Gelfand-Kirillov dimension $\GKdim H$ is finite. 
\item[(2)] $H$ satisfies the Dixmier-Moeglin Equivalence (DME).
\item[(3)] The group $G(H)$ of group-likes of $H$ is nilpotent-by-finite.
\end{enumerate}
\end{conjecture}

Note that, for a pointed Hopf $k$-algebra $H$ (where $k$ is an arbitrary field and $H$ is not necessarily noetherian), $H$ is a faithfully flat $kG(H)$-module by \cite[Theorem~3.2]{Ta}, so that $H$ will be noetherian only if the coradical $kG(H)$ of $H$ is noetherian. There is thus an obvious issue with the conjecture - namely, it is not known which group algebras are noetherian. In the positive direction, Philip Hall \cite{H}, \cite[Corollary~10.2.8]{P} adapted the proof of the Hilbert Basis Theorem in 1954 to show that $kG$ is noetherian when $G$ is polycyclic-by-finite, and to date these remain the only known examples of noetherian group algebras. 

Hence our first objective here is to address the basic case of Conjecture~\ref{wild} where $H= kG$, the group algebra of a polycyclic-by-finite group $G$. 
For group algebras (of \emph{all} finitely generated groups) the implications $(1)\Longleftrightarrow (3)$ are known thanks to famous theorems of Bass, Guivarc'h and Gromov \cites{B, Gu, G}. Zalesskii \cite{Z} showed in 1971 that when $G$ is finitely generated and nilpotent-by-finite every primitive ideal of $kG$ is maximal, from which the DME follows easily, as was shown in \cite{L}; see also \cite[Theorem~5.3]{Be}. Therefore, for $H = kG$ with $G$ polycyclic-by-finite and with $k$ \emph{any} field, the implications $(1)\Leftrightarrow (3) \Rightarrow (2)$ are already known. But although Lorenz exhibited in \cite{L} a polycyclic group whose complex group algebra fails to satisfy the DME, it has remained unclear whether $(2) \Longrightarrow (3)$ for group algebras of polycyclic-by-finite groups. We rectify this omission in $\S$\ref{subsect2.2}, by proving:

\begin{theorem}\label{DMEthm} {(See Theorem~\ref{main}.)} Let $H$ be the group algebra of a polycyclic-by-finite group over a field $k$ (of any characteristic). Then Conjecture~\ref{wild} holds for $H$.
\end{theorem}

\medskip 

The implication $(1)\Longrightarrow (2)$ of Conjecture~\ref{wild} is proposed, without the pointed hypothesis but for $k = \mathbb{C}$, as \cite[Conjecture~1.3]{BL}, and the cocommutative case of this implication is obtained as \cite[Theorem~1.4]{BL}. Combining this result with Theorem~\ref{DMEthm} and the Cartier-Gabriel-Kostant structure theorem for cocommutative Hopf $k$-algebras in characteristic 0 \cite[\S\S5.6.4--5.6.5]{M}, we prove that Conjecture~\ref{wild} is true for all currently known cocommutative noetherian Hopf algebras: 

\begin{corollary}\label{cocomcor-intro} {\rm(}See Corollary~\ref{cocomcor}{\rm)} Let $k$ be an algebraically closed field of characteristic 0 and let $H$ be a cocommutative Hopf $k$-algebra. Assume that the group $G(H)$ of group-likes of $H$ is polycyclic-by-finite and that the Lie algebra $\mathfrak{t}$ of primitive elements of $H$ has finite dimension. Then Conjecture~\ref{wild} is true for $H$.
\end{corollary}
\medskip
 
In our second topic, which is the focus of $\S$\ref{finiteness}, we weaken the noetherian condition by studying when a group algebra $kG$ is a
 (semi)prime Goldie ring; equivalently when $kG$ has a (semi)simple artinian ring of fractions $Q(kG)$. Building on recent beautiful results of Bartholdi, Kielak, 
 Kropholler and Lorensen \cites{BK, KrL}, we prove the following result. See Theorem~\ref{done} for an expanded version of this result and 
 $\S$\ref{finiteness} for unexplained terminology.
 
\begin{theorem}\label{introthm2} Let $G$ be a group and $k$ an algebraically closed field of characteristic $0$. 
Consider the following statements:
 \begin{enumerate}
\item[(a)] $G$ is amenable and there is a bound on the orders of finite subgroups of $G$.
\item[(b)] $kG$ has finite (right) uniform dimension.
 \item[(c)] $Q(kG)$ exists and is semisimple artinian.
\item[(d)] $G$ is elementary amenable and there is a bound on the orders of finite subgroups of~$G$.
\end{enumerate}
Then
 $\qquad \qquad (d) \Longrightarrow (c) \Longleftrightarrow (b)\Longrightarrow (a).$
\end{theorem}

There is an analogous but slightly more complicated result (Theorem~\ref{pdone}) in positive characteristic. 

 One consequence of these two results is that, if the zero divisor conjecture has a negative answer for amenable groups, then there are counterexamples that are very far from domains. For example, in Proposition~\ref{zerodivamen} we prove the following result.

\begin{proposition}\label{zerodivamen-intro} 
Assume that there exists a torsion free amenable group $H$ and a field $k$ such that $kH$ is not a domain. Then there is a finitely generated torsion free amenable group $G $ such that $kG$ has infinite uniform dimension. 
\end{proposition}
\medskip 

In $\S$\ref{noetherian} we address our third topic, a question alluded to above, namely: which group algebras $kG$ are noetherian? Recall that every (right) noetherian ring $R$ has (right) (Gabriel-Rentschler-Krause) Krull dimension, $\Kdim R$, in the sense of \cite[Chapter~6]{McCR}. When $G$ is polycyclic-by-finite $\Kdim kG <\infty$, 
so we ask whether a group algebra $kG$ is noetherian with finite Krull dimension only if $G$ is polycyclic-by-finite; in fact we are bold enough to propose a positive answer as Conjecture~\ref{Krullqn}. The additional hypothesis of finite Krull dimension opens the door to a proof by induction. We do not prove this conjecture, but we show that a minimal counterexample $\widehat{G}$ to a positive answer is quite strongly constrained. This is given in Theorem~\ref{summary}, with an abbreviated version as follows. Here, a \emph{just infinite} group is an infinite group with all its proper factor groups finite, while a \emph{hereditarily just infinite} group is a residually finite group in which every subgroup of finite index is just infinite. 

\begin{theorem}\label{shortnoeth} Let $k$ be an algebraically closed field with $\cchar k = 0$. Assume that $n$ is minimal such that there exists a group $G$ which is not polycyclic-by-finite but with $kG$ noetherian of finite Krull dimension $n$. Then there exists a group $\widehat{G}$ with the same properties, such that
\begin{enumerate}
\item $\widehat{G}$ is amenable but not elementary amenable;
\item $\widehat{G}$ satisfies the ascending chain condition (ACC) on subgroups and there is a bound on the orders of its finite subgroups;
\item $\widehat{G}$ is just infinite, and is either $(a)$ hereditarily just infinite or $(b)$ simple;
\item if $\widehat{G}$ is not simple, then it has no infinite torsion subgroups.
\end{enumerate}
\end{theorem}

Once again, a slightly weaker result holds in characteristic $p>0$; see Theorem~\ref{summary}.
\bigskip

\subsection{Notation}\label{notation} Throughout, $k$ will denote an arbitrary field, with additional hypotheses on $k$ made explicit when required. Recall that a field $k$ is \emph{absolute} if it is an algebraic extension of a finite field.

The class of polycyclic-by-finite groups will be denoted by $\mathcal{P}$; thus, $G \in \mathcal{P}$ if and only if $G$ has a finite series of subgroups 
\begin{equation}\label{Pchain} 1 = H_0 \subset H_1 \subset \cdots \subset H_n = G
\end{equation}
 with $H_i \triangleleft H_{i+1}$ for all $i$ and each subfactor $H_{i + 1}/H_i$ either cyclic or finite. If $G \in\mathcal{P}$, the \emph{Hirsch number} $h(G)$ is the number of infinite cyclic factors in a chain \eqref{Pchain}.
 Note that $G\in \mathcal{P}$ has a poly-(infinite cyclic) characteristic subgroup of finite index \cite[Lemma~10.2.2]{P}. The class of groups satisfying the ascending chain condition (ACC) for subgroups is denoted $\mathrm{Max}$. It is easy to see that the solvable groups with $\mathrm{Max}$ are the polycyclic groups.

If $T$ is a subgroup of a group $\Gamma$ and $S$ is a subset of $\Gamma$ we denote the \emph{centraliser} of $S$ in $T$ by $\Cent_T(S)$, and the \emph{normaliser} of $S$ in $T$ by $N_T(S)$; that is,
$$ \Cent_T(S) \; = \; \{ t \in T : tst^{-1} = s \quad\forall s\in S \}, \textit{ and } N_T(S) \; = \; \{t \in T \, : \, tst^{-1} \in S \quad\forall s \in S\}. $$ 
The \emph{FC-subgroup} of $T$ is denoted by $\Delta(T)$; that is, 
$$ \Delta (T) \; := \; \{t \in T : |T \, : \, \Cent_T(t)| < \infty \},$$
the characteristic subgroup of $T$ composed of those $t\in T$ with only finitely many conjugates. The elements of finite order in $\Delta(T)$ form a characteristic locally finite subgroup, \emph{ the torsion $FC$-subgroup} $\Delta^+ (H)$ of $H$. Moreover,
 $\Delta(T)/\Delta^+ (T)$ is torsion-free abelian. See \cite[$\S$4.1]{P} and in particular by \cite[Lemma~4.1.6]{P} for more details. 
 
The \emph{Gelfand-Kirillov dimension} of a $k$-algebra $R$, resp. of an $R$-module $M$, is denoted by $\GKdim R$ resp.
 $\GKdim M$. Our reference for the Gelfand-Kirillov dimension is \cite{KL}. Further, $\mathrm{Spec}(R)$ denotes the space of prime ideals of $R$.
 
 Given a Hopf algebra $H$ we denote the group of group-likes of $H$ by $G(H)$, and the space of primitive elements of $H$ by $P(H)$. 
\bigskip

\section{The Dixmier-Moeglin Equivalence}\label{group}

In this section we study the Dixmier-Moeglin Equivalence for group rings of polycyclic-by-finite groups, proving Theorem~\ref{DMEthm} and Corollary~\ref{cocomcor-intro}. We begin with the relevant definitions.

\begin{definition}\label{DMEdefn} Let $R$ be a noetherian $k$-algebra.
\begin{itemize}
\item[(i)] A prime ideal $P$ of $R$ is \emph{rational} if the centre of the Goldie quotient ring of $R/P$ is an algebraic extension of $k$.
\item[(ii)] $R$ satisfies the \emph{Dixmier-Moeglin
Equivalence} (DME) if for every $P \in \mathrm{Spec}(R)$ the following properties are equivalent:
\begin{enumerate}
\item[(A)] $P$ is primitive;
\item[(B)] $P$ is rational;
\item[(C)] $P$ is locally closed in $\mathrm{Spec}(R)$ in the Zariski topology.
\end{enumerate}
\end{itemize}
\end{definition}

Recall that a noetherian $k$-algebra $R$ is said to satisfy the \emph{Nullstellensatz} if every prime ideal is an intersection of primitive ideals and, moreover, the endomorphism algebra of every simple $R$-module is algebraic over $k$. Many noetherian algebras satisfy the Nullstellensatz; for example any affine algebra over an uncountable 
field $k$ \cite[Propositions~9.1.6 and ~9.1.7]{McCR}. Moreover, if the noetherian algebra $R$ satisfies the Nullstellensatz then the implications $(C) \Longrightarrow (A) \Longrightarrow (B)$ of Definition~\ref{DMEdefn}(ii) always hold \cite[Lemma~II.7.15]{BG}.
 Since the ground-breaking work of Dixmier and Moeglin \cites{D, M} who proved that the enveloping algebra $U(\mathfrak{g})$ of every 
 finite dimensional complex Lie algebra $\mathfrak{g}$ satisfies the DME, it has been shown to hold for many important classes of noetherian algebras. A short
 survey with detailed references is given in \cite{Be}; see also the summary in \cite[pp.~1844-1845]{BL}. 


\subsection{On the DME for group algebras of polycyclic-by-finite groups}\label{subsect2.2}

The objective in this subsection is Theorem~\ref{main}, which is a more precise version of Theorem~\ref{DMEthm}. 
Recall that $\mathcal{P}$ denotes the class of polycyclic-by-finite groups.

\begin{lemma}\label{FClem} Let $G \in \mathcal{P}$. Then $\Delta (G)$ contains a free abelian subgroup $A$ of finite index which is normal in $G$.
\end{lemma}
\begin{proof} Since $G$ satisfies ACC on subgroups, $\Delta (G)$ is finitely generated. Therefore, by \cite[Lemma~4.1.5]{P}, the 
normal subgroup $\Delta^+(G)$ of $\Delta (G)$ is finite, with $\Delta (G)/\Delta^+(G)$ free abelian of finite rank. 
Since $G$ is polycyclic-by-finite it is residually finite \cite[Lemma~10.2.11]{P}, so there exists a normal subgroup $M$ of finite index in
 $G$ such that $M \cap \Delta^+(G) = \{1\}$. Put $A := M \cap \Delta (G)$, so that $A$ is normal in $G$, has finite index in $\Delta (G)$ and embeds
 in $\Delta (G)/\Delta^+(G)$, hence is free abelian. 
\end{proof}

\begin{lemma}\label{FClem2}Let $F$ be a finite normal subgroup of a group $T$. Then $\Delta(T/F) = \Delta(T)/F$. In particular, if $\Delta (T)$ is finite then the FC-subgroup of $T/\Delta (T)$ is trivial.
\end{lemma}

\begin{proof} Let $x \in T$ and let $\{x_i F \, : \, i \in \mathcal{I}\}$ constitute the $T/F$-conjugates of $xF$. The lemma is immediate from the fact that the $T$-conjugates $\euls{C}_T(x)$ of $x$ constitute a subset of $ \bigcup_{i \in \mathcal{I}} x_i F $ with $\euls{C}_T (x) \cap x_i F \neq \emptyset$ for each $i \in \mathcal{I}$.
\end{proof}

To prove Theorem~\ref{main} we first prove that group algebras of polycyclic-by-finite groups are residually simple artinian whenever they satisfy the (obviously necessary) requirement of being semiprime. For this we need some definitions.
Given a prime $q$, a finite group $F$ is called $q$-\emph{nilpotent} if $F$ has a normal subgroup $N$ of order prime to $q$, with $F/N$ a $q$-group. 
A polycyclic-by-finite group $G$ is then called \emph{$q$-nilpotent } if all its finite images are $q$-nilpotent. 

\begin{lemma}\label{caught} Let $G \in \mathcal{P}$ and $k$ be a field. If $k$ has characteristic $p>0$, assume in addition that $\Delta(G)$ contains no elements of order 
$p$. Then $kG$ is residually simple artinian; that is, 
\begin{equation}\label{tiny} 0\; = \; \bigcap \{ P \, : \, P \triangleleft kG, \; P \textit{ maximal, }\mathrm{dim}_k(kG/P) < \infty \}.
\end{equation} 
\end{lemma}

\begin{proof} If $k$ has characteristic 0 the result follows immediately from the fact that $G$ is residually finite, \cite[Lemma~10.2.11]{P}, together with Maschke's Theorem. We may therefore suppose that $k$ has characteristic $p > 0$. Fix a prime $q$ with $q \neq p$. We first prove:

\begin{sublemma}\label{sublem} Keep the notation as above. Then there exists a characteristic torsion-free subgroup $Q$ of finite index $n$ in $G$ with $Q$ a residually finite $q$-group.
\end{sublemma}

\begin{proof} By a result of Roseblade, \cite[Lemma~11.2.16]{P}, every polycyclic-by-finite group contains a characteristic $q$-nilpotent subgroup $Q$ of finite index. Since polycyclic-by-finite groups are poly-(infinite cyclic)-by-finite, and subgroups of finite index in $q$-nilpotent groups are again $q$-nilpotent \cite[proof of Lemma~11.2.16]{P}, we may choose such a subgroup $Q$ which is also torsion-free. Then, by \cite[Corollary~2.5]{CS}, $Q$ is residually a finite $q$-group. 
\end{proof}

 Returning to the proof of the lemma, pick a subgroup $Q$ as in Sublemma~\ref{sublem}. By Maschke's Theorem applied to the group algebras $kF$ of the finite $q$-group images $F$ of $Q$,
\begin{equation}\label{zero} 0 \; = \; \bigcap \{ M\, : \, M \triangleleft kQ, \; M \textit{ maximal, }\mathrm{dim}_k(kQ/M) < \infty \}. 
\end{equation}
Let $\mathcal{M}$ be the set of co-artinian maximal ideals of $kQ$ and, given $M \in \mathcal{M}$, set
 $\widehat{M} := \bigcap_{g\in G} M^g $. This is a finite intersection, so that there is a crossed product decomposition
\begin{equation}\label{crossed} R_{\widehat{M}}\; := \; kG/\widehat{M}kG \; \cong \; (kQ/\widehat{M})\ast (G/Q),
\end{equation}
with $\mathrm{dim}_k(R_{\widehat{M}}) < \infty$. By \eqref{zero} and the fact that $kG$ is a free $kQ$-module,
\begin{equation}\label{contain} 0 \; = \; \bigcap \{ \widehat{M}kG\, : \, M \in \mathcal{M} \}.
\end{equation}
Note that $kQ/\widehat{M}$ is semisimple and $R_{\widehat{M}}$ is generated as a $kQ/\widehat{M}$-module by a normalising set of $n := |G : Q|$ elements, namely the images in $R_{\widehat{M}}$ of a set of coset representatives of $Q$. Hence, if $J(R_{\widehat{M}})$ denotes the Jacobson radical of $R_{\widehat{M}}$, it follows from \cite[Theorem~7.2.5]{P} that
\begin{equation}\label{index} J(R_{\widehat{M}})^n \; = \; 0\qquad \text{for all $M \in \mathcal{M}$}.
\end{equation}
Denote the right side of \eqref{tiny} by $I$. By the yoga of prime ideals in crossed products of finite groups, as in for example \cite[$\S$14]{P3}, every maximal ideal $P$ occurring in the definition of $I$ features as a maximal ideal $P/\widehat{M}kG$ of an algebra $R_{\widehat{M}}$ as in \eqref{crossed}; indeed, the required ideal $\widehat{M}$ is simply $P \cap kQ$. Hence, by \eqref{index} and \eqref{contain},
\begin{align*} I^n \; &= \; \left( \bigcap \{ P \, : \, P \triangleleft kG, \; P \textit{ maximal, }\mathrm{dim}_k(kG/P) < \infty \}\right)^n\\
&\subseteq \; \bigcap \{ \widehat{M}kG\, : \, M \triangleleft kQ, \; M \textit{ maximal, }\mathrm{dim}_k(kQ/M) < \infty \}\\
&= \; 0.
\end{align*}
But we are assuming that there is no $p$-torsion in $\Delta (G)$. Thus \cite[Theorem~4.2.13]{P} implies that $kG$ is semiprime, whence $I^n=0$ implies that $I = 0$, as required.
\end{proof}

We are now ready to prove our main result on the DME, thereby proving Theorem~\ref{DMEthm} from the introduction.

\begin{theorem}\label{main} Let $k$ be a field and let $G \in \mathcal{P}$ with $G$ not nilpotent-by-finite. Then $kG$ fails to satisfy the Dixmier-Moeglin Equivalence. More precisely, 
\begin{enumerate}
\item[(i)] if $k$ is not an absolute field, then $kG$ has a primitive and rational ideal which is not locally closed;
\item[(ii)] if $k$ is an absolute field, then $kG$ has a rational ideal which is neither primitive nor locally closed.
\end{enumerate}
\end{theorem}

\begin{proof}
Consider the following collection of ordered pairs of subgroups of $G$:
$$ \euls{B} \; := \; \{(N,B) : B \triangleleft G,\; |G/B| < \infty, N = \textit{ terminus of upper central series of }B\}. $$
Note that given $B \triangleleft G$ with $|G/B| < \infty$ then there exists $N$ such that $(N,B) \in \euls{B}$, since $G \in \mathrm{Max}$. 
Moreover $N$ is a characteristic subgroup of $B$, so that $N \triangleleft G$. Again since $G \in \mathrm{Max}$, we can choose 
$(N_0,B_0) \in \euls{B}$ such that $N_0$ is a maximal member of the set $\{N \, :\, \exists\, B \textit{ such that } (N,B) \in \euls{B}\}$. 
Fix this pair $(N_0,B_0)$.

\medskip

Observe that 
\begin{equation}\label{one} G/N_0 \textit{ is not nilpotent-by-finite.}
\end{equation}
For suppose that $L/N_0$ is a normal nilpotent subgroup of finite index in $G/N_0$. Then 
$$N_0\ \triangleleft\ L \cap B_0 \ \triangleleft \ G,$$
and $N_0$ is contained in the upper central series of $L \cap B_0$. Therefore $L \cap B_0$ is a normal nilpotent subgroup of finite index in $G$, contradicting our hypothesis on $G$. We now claim that 
\begin{equation}\label{two} |\Delta (G/N_0)| < \infty. 
\end{equation}
For suppose that \eqref{two} is false. Then by Lemma~\ref{FClem} applied to $G/N_0$ we can find a torsion free abelian subgroup $A/N_0$ of finite index in $\Delta(G/N_0)$ with $A \triangleleft G$. Define a subgroup $C$ of $G$ with $N_0 \subseteq C$, by 
$$ C/N_0 \; := \; \Cent_{G/N_0}(A/N_0).$$
Then 
$$ A \ \subseteq\ C\ \triangleleft\ G \ \textit{ and } \ |G \, : \, C| \; < \; \infty, $$
the first claim since $A$ is abelian and normal in $G$, and the second since $A/N_0$ is finitely generated and is contained in $\Delta (G/N_0)$. Define $D := B_0 \cap C$, so that
$$ N_0 \ \subset \ \widehat{A} := A \cap D\ \triangleleft \ D\ \triangleleft \ G, \; \textit{ with }\; |\widehat{A} : N_0| = \infty \textit{ and } \; |G : D |< \infty. $$
Note that $\widehat{A}$ is contained in the upper central series of $D$, since $N_0$ is and $\widehat{A}/N_0 \subseteq Z(D/N_0)$. Since $\widehat{A}/N_0$ is infinite the pair $(\widehat{A},D)$ contradicts the maximality of $N_0$, proving \eqref{two}.

Define $F$ to be the normal subgroup of $G$ such that $N_0 \subseteq F$ and $F/N_0 = \Delta (G/N_0)$. By \eqref{two}, $|F : N_0| < \infty$, so that $G/F \neq \{1\}$ by \eqref{one}, and $\Delta(G/F) = \{1\}$ by Lemma~\ref{FClem2}. 

\medskip 
We now consider cases $(i)$ and $(ii)$ separately. 

$(i)$ Since $k$ is not absolute, $k(G/F)$ is primitive by \cite[Theorem~F1]{R}. Since $\Delta(G/F) = \{1\}$, the centre of the Goldie quotient ring of $k(G/F)$ is $k$ by a result of Formanek 
\cite[Theorem~4.5.8]{P}. In other words, $\mathfrak{f}kG$ is a rational and primitive ideal of $kG$, where $\mathfrak{f}$ denotes the augmentation ideal of $kF$. On the other hand, $\Delta(G/F) = \{1\}$, so Lemma~\ref{caught} implies that $\mathfrak{f}kG$ is an intersection of co-artinian maximal ideals of $kG$. Thus
 $\mathfrak{f}kG$ is not locally closed.

$(ii)$ Suppose that $k$ is an absolute field. Since $\Delta(G/F) = \{1\}$, Connell's Theorem \cite[Theorem~4.2.10]{P} implies that $k(G/F)$ is prime. Hence $\mathfrak{f}kG$ is a rational prime ideal of $kG$, where rationality follows as in $(i)$. But, by Lemma~\ref{caught}, $\mathfrak{f}kG$ is not locally closed and it is not primitive since, by \cite[Theorem~12.3.7]{P}, the simple $k(G/F)$-modules are all finite dimensional.
\end{proof}

\medskip

\subsection{Cocommutative Hopf algebras}\label{cocomsec}

If one considers Conjecture~\ref{wild} for cocommutative Hopf algebras two glaring obstacles quickly appear: the problem of determining which group algebras are noetherian and whether the only enveloping algebras $U(\mathfrak{t})$ that are noetherian are those for which $\dim_k\mathfrak{t}<\infty$. 
 Although there has been significant progress on the latter question in recent years \cites{SW, Bu, AM}, it remains open. If we sidestep these two problems, then Theorem~\ref{main} easily implies Conjecture~\ref{wild} for cocommutative Hopf algebras:

\begin{corollary}\label{cocomcor} Let $k$ be a field of characteristic 0 and let $H$ be a cocommutative Hopf $k$-algebra, where 
 either $k$ is algebraically closed or $H$ is pointed. Assume moreover 
that the group $G(H)$ of group-likes of $H$ is polycyclic-by-finite and that the Lie algebra $\mathfrak{t}$ of primitive elements of $H$ has finite dimension. Then Conjecture~\ref{wild} holds for $H$.
\end{corollary}

\begin{proof} If $k$ is algebraically closed then $H$ is pointed 
 \cite[$\S$5.6, p.~76]{Mo}. 
Thus the Cartier-Gabriel-Kostant structure theorem 
\cite[Corollary~5.6.4(iii) and Theorem~5.6.5]{M} 
applies in both cases and shows that $H$ is a smash 
product $U(\mathfrak{t})\# kG(H)$ where $G(H)$ acts by conjugation on $\mathfrak{t}$. By the Poincar\'e-Birkhoff-Witt
 Theorem and the Hilbert Basis Theorem for skew Laurent extensions \cite[Theorem~1.45, Proposition~1.7.14]{McCR} it follows that $H$
 is noetherian under our stated hypotheses on $G(H)$ and $\mathfrak{t}$. 
 We now show that the properties $(1), (2)$ and $(3)$ listed in Conjecture~\ref{wild} are equivalent for $H$. 

\medskip

\noindent $(1) \Longrightarrow (2)$: This is \cite[Theorem~1.3]{BL}.

\noindent $(2) \Longrightarrow (3)$: There is a homomorphism of Hopf algebras from $H$ onto $kG(H)$. Therefore, if $H$ satisfies the DME, so must $kG(H)$. Hence, by Theorem~\ref{main}, $G(H)$ is nilpotent-by-finite.

\noindent $(3) \Longrightarrow (1)$: Assume that $G(H)$ is nilpotent-by-finite, and note that it is finitely generated. Since $G(H)$ acts on $U(\mathfrak{t})$ by automorphisms of $\mathfrak{t}$, $H$ has finite GK-dimension by the argument used in \cite[proof of Corollary~3.4]{BL}.
\end{proof}

\begin{question} \label{too-wild} Does the analogue of Conjecture~\ref{wild} also hold for arbitrary noetherian Hopf algebras defined 
 over arbitrary fields? The authors suspect that the answer is ``No'' but know of no such examples. \end{question}
 
 When trying to answer Question~\ref{too-wild} it may be worthwhile initially to consider \emph{affine} noetherian Hopf algebras. The additional hypothesis is redundant for pointed Hopf algebras thanks to \cite[Theorem~1.5(3)]{GZ}. It is an open question whether all noetherian Hopf $k$-algebras are affine.
\bigskip

\subsection{Beyond Hopf}
We end the section by noting that there is no general relationship between finite Gelfand-Kirillov dimension and the Dixmier-Moeglin equivalence for finitely generated algebras that are not Hopf algebras. 
 
\begin{ex}\label{non-wild}
\emph{Let $k$ be a field and let $A$ be an affine noetherian $k$-algebra. Consider the statements:
\begin{enumerate}
\item[(1)] $\GKdim A<\infty$;
\item[(2)] $A$ satisfies the DME.
\end{enumerate}
Then the implications $(1)\Longrightarrow (2)$ and $(2)\Longrightarrow (1)$ are both false.} 

First, for each positive integer $n\ge 4$, an example of a finitely generated noetherian algebra of GK dimension $n$ that does not satisfy the DME is given in \cite[Theorem~9.1]{BLLM}.

Conversely, one can construct an example of a noetherian algebra of exponential growth that satisfies the DME as follows.
Let $q\in \mathbb{C}$ be transcendental and let $R$ be the skew Laurent ring $R = \mathbb{C}[x^{\pm 1}][y^{\pm 1};\tau]$ where $\tau(x)=qx$ (thus $yx=qxy$). Then $R$ is a simple ring by \cite[Example~1.8.6]{McCR}.
Observe that $(x^2y)(x^3y) = q (x^3y)(x^2y)$ and so we have a $\mathbb{C}$-algebra automorphism $\sigma$ of $R$ induced by $x\mapsto x^3y$, $y\mapsto x^2y$.
Then a simple exercise shows that $A: = R[z^{\pm 1};\sigma]$ has exponential growth since $z^n x z^{-n}= x^{\alpha(n)}y^{\beta(n)}$ where $\alpha(n)$ grows exponentially. As such, no power of $\sigma$ can be inner and $A$ is simple by \cite[Theorem~1.8.5]{McCR}. 

This immediately implies that $A$ satisfies the Dixmier-Moeglin equivalence. Indeed, the zero ideal $(0)$ is certainly locally closed and hence primitive. Moreover $(0)$ is rational since $A$ satisfies the Nullstellensatz (see \cite[Theorem~9.1.8]{McCR}). 

\end{ex}

\section{Finiteness conditions for group algebras}\label{finiteness}

In this section we extend results of Bartholdi, Kielak, 
 Kropholler and Lorensen \cites{BK, KrL} to examine the relationship between (elementary) amenability of a group and the Goldie conditions on its group ring $kG$. 
 The details are given in Theorems~\ref{done} and~\ref{pdone}, for the cases where $k$ has characteristic zero and $p > 0$, respectively. Various applications are given at the end of the section. 
 
We begin with one of many equivalent definitions of the amenability condition on a group. A nice short survey of this topic can be found in \cite[$\S$8]{G2}; for a more detailed account, see \cite{Jus}.

\begin{definition}\label{amenable} Let $G$ be a group, $\euls{S}(G)$ the set of subsets of $G$.
\begin{enumerate}
\item[(i)] A \emph{finitely additive invariant probability measure on $G$} is a map $\mu : \euls{S}(G) \rightarrow [0,1]$ such that
\begin{itemize}
\item[(a)] $\mu (G) = 1$;
\item[(b)] for all subsets $A$ and $B$ of $G$ with $A \cap B = \emptyset$, $\mu(A \cup B) = \mu(A) + \mu (B)$;
\item[(c)] for all subsets $A$ of $G$ and all $g \in G$, $\mu(A) = \mu (gA) = \mu (Ag)$.
\end{itemize}
\item[(ii)] $G$ is \emph{amenable} if it has a finitely additive invariant probability measure.
\item[(iii)] The class of \emph{elementary amenable} groups is the smallest class of groups containing $\mathbb{Z}$ and all finite groups, and closed under taking subgroups, quotients, extensions, and directed unions (an alternative description is given in the proof of Lemma~\ref{elamenMax}).
\end{enumerate}
\end{definition}

As the name suggests, every elementary amenable group is amenable, but the converse is false as first demonstrated by the groups of intermediate growth constructed by Grigorchuk \cite{G3}. A criterion for amenability, key for us here, is the following result.

\begin{theorem}\label{small}{\rm(Kropholler-Lorensen, \cite[Theorem~A]{KrL})} Let $k$ be a field and $G$ a group. Then the following are equivalent:
\begin{enumerate}
\item[(i)] For every positive integer $n$ there does not exist an embedding of right $kG$-modules
\begin{equation}\label{into} kG^{\oplus (n+1)} \ \hookrightarrow \ kG^{\oplus n}.
\end{equation}
\item[(ii)] $G$ is amenable. \qed
\end{enumerate}
\end{theorem}

Let $k$ be a field and $G$ a group. It is a more or less immediate consequence of Theorem~\ref{small} that if $kG$ is noetherian
 then $G$ is amenable. We shall refine this statement in different ways in the sequel; see for example the implications $(b)\Longrightarrow (a)$ of Theorems~\ref{done} and \ref{pdone} as well as Proposition~\ref{Krullprop}.

 In fact Theorem~\ref{small} also holds for strongly group-graded rings; see \cite[Theorem~A]{KrL} for the precise result. That work is in turn a refinement of arguments of Bartholdi \cite{BK}. A noteworthy aspect of the latter work is the following lovely result of Kielak 
 building on work of Tamari \cite{T}.

\begin{theorem}{\rm(Kielak, \cite[Theorem~A1]{BK})}\label{Ore} Let $k$ be a field and $G$ a group such that $kG$ is a domain. Then $G$ is amenable if and only if $kG$ is an Ore domain. \qed
\end{theorem}

In this section we will be particularly interested in what happens if we drop the domain hypothesis from the above result. For this we need the following definitions and results; see, for example, \cite[$\S$2.2, 2.3]{McCR}.

\begin{definition}\label{Goldie} Let $R$ be a ring and $M$ a right $R$-module.
\begin{enumerate}
\item[(i)] $M$ has \emph{finite uniform dimension} if it contains no infinite direct sum of non-zero submodules. In this case every maximal such direct sum contains the same number of summands, called the \emph{uniform dimension} of $M$, denoted $\udim M $; otherwise, we write $\udim M = \infty$.
\item[(ii)] $R$ is \emph{right Goldie} if $\udim R_R < \infty$ and $R$ has max-ra, the ascending chain condition on right annihilators of subsets of $R$.
\item[(iii)] The \emph{(right) singular ideal} of $R$ is $$\zeta (R)\ :=\ \{ r \in R \, : \, rE = 0, \, E \textit{ an essential right ideal}\}.$$ 
By \cite[\S2.2.4]{McCR}, $\zeta (R)$ is an ideal of $R$.
\item[(iv)] For an ideal $I$ of $R$ set 
$\euls{C}(I) := \, \{x \in R : x + I \textit{ not a zero divisor in } R/I\};$
in particular $\euls{C}(0)$ is the set of regular elements of $R$. 

\end{enumerate}
\end{definition}
By Goldie's Theorem (see Theorem~\ref{Goldiethm}, below) an Ore domain is the same as a Goldie domain. For a group ring $kG$, the right and left Goldie conditions are equivalent, simply because $kG\cong kG^{\rm op}$ via
 the antipode map $g \mapsto g^{-1}$ for $g \in G$. For the same reason, the right and left uniform dimensions of $kG$ are equal.
We also note the following obvious fact.

\begin{lemma}\label{udimgp} If $H $ be a subgroup of a group $G$ and $k$ is a field, then 
 $\mathrm{udim}(kH) \leq \mathrm{udim}(kG)$.
 \end{lemma}

\begin{proof} Use the fact that $kG$ is a free left $kH$-module.\end{proof}

Since we need the details of Goldie's Theorem, we state the result here.

\begin{theorem}\label{Goldiethm}{\rm(Goldie, \cite[Theorem~2.3.6]{McCR})} Let $R$ be a ring. The following are equivalent.
\begin{enumerate}
\item[(a)] $R$ is semiprime right Goldie.
\item[(b)] $R$ is semiprime, $\zeta (R) = 0$ and $\udim R_R < \infty$.
\item[(c)] $R$ has a right ring of fractions $Q(R)$ with respect to $\euls{C}(0)$ (equivalently, $\euls{C}(0)$ satisfies the right Ore condition) and $Q(R)$ is semisimple artinian.
\end{enumerate}
Moreover $R$ is prime $\Longleftrightarrow$ $Q(R)$ is simple, and in this case $Q(R) \cong M_n(D)$, the ring of $n\times n$ matrices over a division ring $D$, with $n = \udim R$.\qed
\end{theorem}
 
Given a group $G$, define a field $k$ of characteristic $p \geq 0$ to be \emph{big enough for} $G$ if 
$k$ contains a primitive $|F|^{\mathrm{th}}$ root of 1 for every finite subgroup $F$ of $G$ of order coprime to $p$. 
Recall, also, that a $p'$-\emph{group} is a group that has no elements of order $p$
 (in both cases we allow the possibility that $p=0$). 
 
 \begin{theorem}\label{done} Let $G$ be a group and $k$ a field of characteristic $0$. 
Consider the following statements:
 \begin{enumerate}
\item[(a)] $G$ is amenable and there is a bound on the orders of finite subgroups of $G$. 
\item[(b)] $\udim kG < \infty$.
 \item[(c)] $kG$ is right Goldie.
 \item[(d)] $Q(kG)$ exists and is semisimple artinian.
\item[(e)] $G$ is elementary amenable and there is a bound on the orders of finite subgroups of~$G$.
\end{enumerate}
Then the following statements hold.
\begin{itemize}
\item[(i)] $\qquad \qquad (e) \Longrightarrow (d) \Longleftrightarrow (c) \Longleftrightarrow (b)\Longrightarrow (a),$
\\ where, for the second part of $(b) \Longrightarrow (a)$, assume in addition that $k$ is big enough for $G$.
 \item[(ii)] Assume that $(e)$ holds and that $G$ has no non-trivial finite normal subgroups. Then $Q(kG)$ is simple artinian and 
$$\udim kG \; = \; l.c.m.\{|F| \, : \, F \subseteq G, |F| < \infty \}.$$ 

 \end{itemize}
\end{theorem} 

\begin{proof} $(i)$ We will repeatedly use the facts that, as char$\,k=0$, \cite[Theorem~4.2.12]{P} implies that $kG$ is semiprime, while 
\cite[Theorem~4]{S} implies that $\zeta(kG) = \{0\}$. 
\medskip

\noindent $(e) \Longrightarrow (d)$: This is \cite[Theorem~1.2]{KLM}. 
\medskip

\noindent$(d)\Longleftrightarrow(c)$: This is immediate from $(c) \Longleftrightarrow (a)$ of Theorem~\ref{Goldiethm}.
\medskip

\noindent$(c) \Longrightarrow (b)$: See Definition~\ref{Goldie}$(ii)$.
\medskip

\noindent$(b) \Longrightarrow (c)$: As before, $kG$ is semiprime with $\zeta(kG) = \{0\}$. Now use Theorem~\ref{Goldiethm}.

\medskip
\noindent$(b)\Longrightarrow (a)$: Suppose that $\udim kG <\infty$. Then $\udim(kG^{\oplus n}) = n\cdot\udim kG$
 for every positive integer $n$, by \cite[Corollary~2.2.10(iv)]{McCR}. But uniform dimension is non-decreasing under inclusion of modules by \cite[Corollary~2.2.10(iii)]{McCR}, so no embedding of the form \eqref{into} can exist. Hence $G$ is amenable by Theorem~\ref{small}. 

 Suppose for a contradiction that $G$ has finite subgroups of unbounded orders. Let $F$ be a finite subgroup of $G$, and suppose that there are $t$ distinct 
simple $kF$-modules, with dimensions $n_i, \, 1 \leq i \leq t$. Since $k$ is big enough for $F$, Maschke's Theorem and the Artin-Wedderburn Theorem imply that 
\begin{equation}\label{add}\udim kF = \sum_{i=1}^t n_i \; \textit{ and } \; \sum_{i=1}^t n_i^2 = |F|.
\end{equation}
Since the second sums in \eqref{add} are, by hypothesis, unbounded as $F$ ranges through the finite subgroups of $G$, so also the first sums are unbounded. 
Thus, by Lemma~\ref{udimgp} $\udim kG = \infty$, contradicting $(b)$. This proves $(a)$.

 \bigskip
\noindent$(ii)$ This is \cite[Theorem~1.3]{KLM}. 
\end{proof}

\medskip

Let $N(R)$ denote the \emph{nilpotent radical} of a ring $R$. Recall that if $k$ is a field of characteristic $p > 0$ and $G$ is a group then $N(kG)= 0$ if and only if $G$ has no finite normal subgroup of order divisible by $p$ \cite[Theorem~4.2.13]{P}. 
Let $R$ be a ring for which $N(R)$ is nilpotent and $R/N(R)$ is right Goldie, and denote the semisimple artinian quotient ring of $R/N(R)$ by $Q$. As in \cite[\S4.1.2]{McCR}, the \emph{reduced rank} $\rho (M)$ of a right $R$-module $M$ is defined as follows. Take any chain $M = M_0 \supseteq M_1 \supseteq \cdots \supseteq M_n = 0$ of submodules of $M$ such that $M_iN(R) \subseteq M_{i+1}$ for each $i$ and set
$$ \rho(M) \; := \; \sum_{i=0}^{n-1}c\left((M_i/M_{i+1})\otimes_{R/N(R)} Q\right),$$
where $c(X)$ denotes the composition length of a $Q$-module $X$. The notation $\euls{C}(I)$, where $I$ is an ideal of a ring $R$, is explained at Definition~\ref{Goldie}(iv).

 \begin{theorem}\label{pdone} Let $G$ be a group and $k$ a field of characteristic $p > 0$. 
Consider the following statements:
 \begin{enumerate}
\item[(a)] $G$ is amenable and there is a bound on the orders of the finite $p'$-subgroups of $G$. 
\item[(b)] $\udim kG < \infty$.
\item[(c)] $kG/N(kG)$ is right Goldie, $N(kG)$ is nilpotent, $\rho(kG) < \infty$, and $\euls{C}(N(kG)) \subseteq \euls{C}(0)$.
\item[(d)] $Q(kG)$ exists and is artinian.
\item[(e)] $G$ is elementary amenable and there is a bound on the orders of finite subgroups of~$G$.
\end{enumerate}
Then the following hold.
\begin{itemize}
\item[(i)] $\qquad \qquad (e) \Longrightarrow (d) \Longleftrightarrow (c) \Longrightarrow (b)\Longrightarrow (a)$.
\\ where, for the second part of $(b) \Longrightarrow (a)$, assume in addition that $k$ is big enough for $G$.
 
 \item[(ii)] Assume that $(e)$ holds and $G$ has no non-trivial finite normal subgroups. Then $Q(kG)$ is simple artinian and 
$$\udim kG \; = \; l.c.m.\{|F| \, : \, F \subseteq G, |F| < \infty \}.$$
\end{itemize}
\end{theorem}

\begin{proof} $(i)$ $(e)\Longrightarrow (d)$ is again \cite[Theorem~1.2]{KLM}. 

\medskip
 \noindent$(d) \Longleftrightarrow (c)$: This is Small's Theorem, \cite[Theorem~4.1.4]{McCR}.

\medskip
 
\noindent$(c)\Longrightarrow (b)$: Assume that $(c)$ holds. By Small's Theorem, 
 $\rho(kG)$ is the composition length of $Q(kG)$ as a right $Q(kG)$-module. Since 
$$\udim(kG_{kG})\; = \; \udim(Q(kG)_{Q(kG)})\; = \; c\left(\socle(Q(kG)_{Q(kG)})\right),$$ 
 $\rho(kG)$ is an upper bound for $\udim kG $.
\medskip

\noindent$(b)\Longrightarrow(a)$: The argument used to prove $(b)\Longrightarrow(a)$ in Theorem~\ref{done}$(ii)$ also works here.

\medskip 

\noindent $(ii)$ This is again \cite[Theorem~1.3]{KLM}. 
\end{proof}

As a variant of $(d)\Longrightarrow (a)$ in the two theorems we have the following result. 

\begin{lemma}\label{ddone} Let $k$ be a field and $G$ a group such that $kG$ has max-ra (this holds, in particular, if (d) of Theorem~\ref{done} or \ref{pdone} holds). Then $G$ has no infinite locally finite subgroup.
\end{lemma}

\begin{proof} If $G$ contains an infinite locally finite subgroup $L$, 
 one can construct an infinite ascending 
chain of right ideals of $kG$ generated by the augmentation ideals of an ascending chain of finite subgroups $F$ of $L$. 
These right ideals are all annihilator ideals, 
of the elements $\sum_{f \in F}f$, and so $kG$ fails to satisfy max-ra.
\end{proof}

\medskip
 Olshanskii's Tarski monsters, non-cyclic infinite groups with all proper subgroups infinite cyclic, respectively 
of prime order, are constructed in \cite{Ol01}, respectively in \cite{Ol1}, and are proved to be
 non-amenable in \cite[Theorem]{Ols}. The following corollary is thus immediate from $(b)\Longrightarrow (a)$ of Theorems ~\ref{done} and \ref{pdone}. It begs the question of whether $\udim kG = \infty$ for every infinite finitely generated torsion group $G$ and every field $k$.

\begin{corollary}\label{Tarski} 
Let $T$ be a Tarski monster and $k$ a field. Then $\udim kT =\infty.$ \qed
\end{corollary}

\medskip

For the rest of the section we consider 
 other possible implications amongst statements (a)--(e) of Theorems~\ref{done} and \ref{pdone} which are \emph{not} guaranteed by those theorems.
\medskip

We first claim that the implication $(d)\Longrightarrow (e)$ is false in all characteristics. Indeed, in \cite{Gr} Grigorchuk constructs a torsion-free group $G$ of intermediate growth which was shown in \cite[Theorem]{GM} to be right orderable. By \cite[Theorem~3.2]{Ch}, no elementary amenable group can have intermediate growth and hence $G$ is not elementary amenable.
 On the other hand, let $k$ be any field. Since $G$ is right orderable, $kG$ is a domain by \cite[Lemmas~13.1.7 and~13.1.9]{P}. If $kG$ is \emph{not} an Ore domain,
 then by \cite[Theorem]{J}, $kG$ contains a copy of the free $k$-algebra on 2 generators, and so $kG$ has exponential growth; a contradiction. Thus
 $kG$ is an Ore domain and (d) holds. In particular, by Theorem~\ref{Ore}, $G$ is amenable.

In effect, we have proved the following result, with Grigorchuk's example showing that the class of groups to which it applies is non-empty.

\begin{proposition}\label{caught2} Suppose that $G$ is a group of intermediate growth and $k$ a field such that $kG$ is a domain. Then $G$ is amenable but not elementary amenable.\qed
\end{proposition}

\medskip
Observe that the implication $(b)\Longrightarrow(c)$ holds in Theorem~\ref{done}, but is not claimed in 
 the setting of Theorem~\ref{pdone}. In fact it is false in the latter case. For a counterexample, take a field $k$ of characteristic $p>0$, and let $G$ be the lamplighter group $C_p \wr C_{\infty}$, the restricted wreath product of $C_p$ by $C_{\infty}$. Thus $G \cong A \rtimes C_{\infty}$ where $A  = \Pi_{i \in \mathbb{Z}}\langle a_i \rangle$ is an infinite elementary abelian $p$-group, so that $kG \cong kA[X^{\pm 1}; \sigma]$ where $\sigma (a_i) = a_{i+1}$ for all $i$.  Then $kG$ is prime by \cite[Theorem~4.2.10]{P} and semisimple by \cite[Lemma~7.4.12]{P}. Now $kA$ is uniform since it is the direct limit of modular group algebras of finite $p$-groups, each of which has uniform dimension one since it is a scalar local Frobenius algebra. Since $kG$ is an Ore extension of $kA$ it is also uniform by \cite[Theorem~2.7]{Matczuk}. But $kG$ is obviously not Goldie, since $G$ contains an infinite locally finite subgroup and so fails to satisfy max-ra, by Lemma~\ref{ddone}. Thus Theorem~\ref{pdone}(c) fails for $kG$.
 
Another class of examples for which $(b)\Longrightarrow (c)$ fails is found by taking $k$ to have characteristic $p$ and $G$ to be an infinite locally finite $p$-group with no non-trivial finite normal subgroups. For instance $G = C_p \wr C_{p^{\infty}}$ is one such group. Then $kG$ is once again prime by \cite[Theorem~4.2.10]{P}, and uniform for the reason $kA$ was above. But in this case $kG$ is local rather than semisimple, since its augmentation ideal is nil.
\medskip
 
\medskip
We next turn to the implication $(a)\Longrightarrow (b)$: We do not know whether this is true, whatever the characteristic. We therefore ask: 

\begin{question}\label{amenudim} In the settings of Theorem~\ref{done} or \ref{pdone}, does $(a) \Longrightarrow (b)$? 
\end{question}

The following result shows that a positive answer to Question~\ref{amenudim} would imply confirmation of the Zero Divisor Conjecture in characteristic 0 for group algebras of torsion free amenable groups. So, if Question~\ref{amenudim} has a positive answer, it is probably very difficult.

\begin{proposition}\label{zerodivamen} Suppose that $H$ is a torsion free amenable group and $k$ is a field of characteristic 0 such that $kH$ is not a domain. Then there is a finitely generated torsion free amenable group $G$ such that $\udim kG = \infty$. In particular, $kG$ would give a counterexample to $(a)\Longrightarrow (b)$ 
in Theorem~\ref{done}.
\end{proposition}

\begin{proof} Let $k$ and $H$ be as stated. We may assume without loss that $H$ is finitely generated, since the class of amenable groups is subgroup-closed, \cite[\S2.7,p.16]{Jus}. 
 If $\udim kH= \infty$ there is nothing to prove, so we may assume that 
\begin{equation}\label{count}\udim kH \; := t \; < \; \infty.
\end{equation}
Therefore, since $k$ has characteristic 0, $(b)\Longrightarrow(d)$ of Theorem~\ref{done} implies that the Goldie quotient ring $Q(kH)$ exists and is semisimple artinian. Since $H$ is torsion-free $kH$ is prime by \cite[Theorem~4.2.10]{P}, and so Theorem~\ref{Goldiethm} implies that $Q(kH)$ is simple artinian with $Q(kH) \cong M_t(D)$ for a division ring $D$, with $t$ as in \eqref{count}. But now observe that we must have $t > 1$, since by hypothesis $kH$ (and therefore also $Q(kH))$ are not domains.

Now let $G$ be the lamplighter group $H\wr C_{\infty}$. Thus $G$ is the extension $A \rtimes C_{\infty}$ of the direct sum $A$ of countably many copies of $H$ by $C_{\infty}$, so $G$ is still finitely generated and torsion free. Moreover $G$ is amenable because $H$ and $C_{\infty}$ are amenable, and the class of amenable groups is closed under extensions and direct limits \cite[\S2.7, p.17-18]{Jus}. Finally, as $t>1$, the following result implies that $\udim kA=\infty$ and then Lemma~\ref{udimgp} implies that $ \udim kG = \infty.$
\end{proof}

\begin{lemma}\label{prod} Let $k$ be a field and $P$ and $Q$ groups. Then
$$ \udim (k(P \times Q)) \ \geq \ (\udim kP)(\udim kQ).$$
\end{lemma}

\begin{proof} There is an isomorphism of algebras $k(P \times Q) \cong kP \otimes_k kQ$. But if $U$ and $V$ are $k$-algebras and $\sum_{i} I_i$, respectively $\sum_{j} J_j$, are direct sums of non-zero right ideals of $U$, respectively $V$, then $\sum_{i,j} I_i \otimes_k J_j$ is a direct sum of non-zero right ideals of $U \otimes_k V$. 
\end{proof}

\medskip
\begin{remark}\label{bad-char-p} The argument used for Proposition~\ref{zerodivamen} 
does not work in characteristic $p>0$ since one could (in theory) have $\udim kH=1$ and then $\udim kG=1$ as well.
Certainly that sort of behaviour happens for the lamplighter groups discussed after Proposition~\ref{caught2}. \end{remark}
 
\bigskip

\section{Noetherian group algebras}\label{noetherian} 

It is a familiar fact that if $k$ is a field and $G \in \mathcal{P}$, the class of polycyclic-by-finite groups, then $kG$ is a noetherian algebra. 
In this section we consider the converse; namely the famous old question (see e.g. \cite[Question~11.39]{Ko}):
\begin{equation}\label{noethqn} \textit{If the group algebra}\; kG\; \textit{is noetherian, is } G \in \mathcal{P}? 
\end{equation}

\noindent Suppose that $kG$ is a noetherian group algebra. Then, using the freeness of $kG$ as a $kH$-module for each subgroup $H$ of $G$ we can easily see that $G \in \mathrm{Max}$. Moreover, by the Kropholler-Lorensen criterion, Theorem~\ref{small}, $G$ is amenable. There are groups $G \in \mathrm{Max}$ which are not amenable, even torsion-free ones; for example the torsion-free Tarski monsters of Olshanskii \cite{Ol}. But at the time of writing it seems that the only known amenable groups
 with $\mathrm{Max}$ are the polycyclic-by-finite ones.
 
As for elementary amenable groups, the following result is undoubtedly well-known, but we include a proof for completeness.

\begin{lemma}\label{elamenMax} Let $G$ be an elementary amenable group satisfying the ascending chain condition on subgroups. Then
 $G$ is polycyclic-by-finite.
\end{lemma}

\begin{proof} We use ideas from \cite[$\S$3]{KLM}. 
 If $\mathcal{X} $ and $ \mathcal{Y}$ are two classes of groups, let $ \mathrm{L}\mathcal{X}$ denote the class of those groups all of whose finitely 
generated subgroups are in $\mathcal{X}$ and write $\mathcal{X} \mathcal{Y}$ for the class of groups $G$ such that $G$ has a normal 
subgroup $H\in \mathcal{X}$ with $G/H\in\mathcal{Y}$.

The class $\mathcal{E}$ of elementary amenable groups can be constructed inductively as follows. Set
$\mathcal{X}_1=\mathcal{B}$ to
 be the class of finitely generated abelian-by-finite groups. Then set 
$\mathcal{X}_{\alpha} := (\mathrm{L}\mathcal{X}_{\alpha - 1})\mathcal{B} $ if $ \alpha $ is a successor ordinal, while 
$\mathcal{X}_{\alpha} := \bigcup_{\beta < \alpha} \mathcal{X}_{\beta} $ if $ \alpha$ is a limit ordinal. 
Then \cite[Lemma~3.1]{KLM} implies that $\mathcal{E} = \bigcup_{\alpha} \mathcal{X}_{\alpha}$.

Suppose that $G \in \mathcal{X}_{\alpha} \cap \mathrm{Max}$ and the result has been proved for all smaller ordinals. Since $G$ is finitely 
generated we may assume that $\alpha$ is not a limit ordinal. Thus $G$ has a normal subgroup $H$ such that $G/H \in \mathcal{B}$ and $H \in \mathcal{X}_{\alpha - 1}$. Since $H\in \mathrm{Max}$, $H\in \mathcal{P}$ by induction. On the other hand, $G/H$ is finitely generated and hence
 abelian-by-finite. Thus $G\in \mathcal{P}$.
 \end{proof}
\medskip

 The (right) Krull dimension of a ring $R$ or an $R$-module $M$ will be denoted by $\Kdim R $, resp. $\Kdim M$. Basic properties of the Krull dimension can be found in \cite[Chapter~6]{McCR}. As is proved in \cite[Lemma~6.2.3]{McCR}, every right noetherian ring $R$ admits a Krull dimension $\Kdim R$, although in general this can be an infinite ordinal. 
 
 Let's approach the question stated as \eqref{noethqn} by considering a perhaps easier special case, which we are foolhardy enough to formulate as a conjecture. For completeness we state also the restriction to torsion groups, which also remains open.

\begin{conjecture}\label{Krullqn} Let $k$ be a field and $G$ a group with $kG$ noetherian and $\Kdim kG < \infty$.
\begin{enumerate}
\item[(i)] $G \in \mathcal{P}$. 
\item[(ii)] If $G$ is a torsion group then $G$ is finite.
\end{enumerate}
\end{conjecture}
 
As a small piece of evidence in favour of part (ii) of the conjecture, we note that it does hold 
 for group algebras of $2$-groups in characteristic 2, even without the restriction on Krull dimension. 

\begin{proposition}\label{2-groups} Let $G$ be a torsion $2$-group such that $kG$ is noetherian for some field $k$ of characteristic $2$. Then $|G|<\infty.$
\end{proposition}

\begin{proof} Since
 $k G $ is noetherian, clearly $G$ is finitely generated. Set $N=N(kG)$ for the nilpotent radical. 
 Then, by Goldie's Theorem (Theorem~\ref{Goldiethm})
 $R :=(kG)/N$ has semi-simple artinian quotient ring $Q(R)$. By the Artin-Wedderburn Theorem
$$Q(R)\ \cong \ \prod_{i=1}^s M_{n_i}(D_i)$$ where $n_1,\ldots ,n_s$ are positive integers, and $D_1,\ldots ,D_s$
 are division algebras of characteristic $2$. 
The map $G\to Q(R)$ induces a group homomorphism $\phi_i$ from $G$ to ${\rm GL}_{n_i}(D_i)$ for each $i $ and 
we set $G_i=\phi_i(G)$. By \cite[Corollary~2]{DW}, $|G_i|<\infty$ and hence the image of $G$ in 
$Q(R)$ is also finite. Thus $R$ is a finite dimensional $k$-algebra. But now, since every power of $N$ is finitely
 generated as a right ideal of $k G$, the factors $N^i/N^{i+1}$ are finitely generated $R$-modules and hence finite dimensional
 for every $i$. Since $N$ is nilpotent it follows that $G$ is finite.
\end{proof}

 Return now to the general Conjecture~\ref{Krullqn}. Restricting our focus to group algebras of finite Krull dimension makes sense for at least three reasons. 

First, the known noetherian group algebras have finite Krull dimension. Indeed, let $h(G)$ denote the Hirsch number of a 
group $G$, as defined in $\S$\ref{notation}. If $G \in \mathcal{P}$ then $\Kdim kG = h(G)$, \cite{Sm}, \cite[Proposition~6.6.1]{McCR}. 

Our second justification is the following result. 

\begin{proposition}\label{Krullprop} Let $k$ be a field and $G$ a group such that $\Kdim kG$ exists. Then $G$ is amenable.
\end{proposition}

\begin{proof} Suppose that $\Kdim kG$ exists. Then the right $kG$-module $kG$ has finite uniform 
dimension by \cite[Lemma~6.2.6]{McCR}. By $(b) \Longrightarrow (a)$ of Theorems~\ref{done} and \ref{pdone}, $G$ is amenable.
\end{proof}
 
Our third reason for invoking Krull dimension is practical: if $\Kdim kG = 0$ then $kG$ is artinian, and a group algebra $kG$ is 
artinian only if $G$ is finite \cite[Theorem~10.1.1]{P}. So we have a starting point for a proof by induction. The target of the rest of this
 section is thus Theorem~\ref{summary}, giving properties of a minimal counterexample $G$ to Conjecture~\ref{Krullqn}.

Recall that a group is \emph{just infinite} if it is infinite but all its proper factors are finite.

\begin{proposition}\label{mincrim} Let $k$ be a field. Let $n \in \mathbb{Z}_{>0}$ and suppose that every group $F$ such that $kF$ is noetherian with $\Kdim kF < n$ is in $\mathcal{P}$. Suppose that there exists a group $H \notin \mathcal{P}$ such that $kH$ is noetherian with 
$\Kdim kH = n$.

 Then there exists a just infinite group $G$ with these properties. Moreover $G$ is a subfactor of $H$ and $kG$ is prime.
\end{proposition}

\begin{proof} Let $H$ be as stated in the proposition. Since $H \in \mathrm{Max}$ and $\Kdim\overline{R} \leq \Kdim R$ for all factor rings $\overline{R}$ of a noetherian ring $R$, we can replace $H$ by a proper factor if necessary so that
\begin{equation}\label{key} \textit{ every proper quotient of $H$ is in }\mathcal{P}.
\end{equation}
Next we prove that
\begin{equation}\label{prime} \textit{if } 1 \neq N \triangleleft H \textit{ then } H/N \in \mathcal{P} \textit{ and } kN \textit{ is prime.}
\end{equation}
The first claim in \eqref{prime} follows from \eqref{key}. For the second, recall the torsion FC subgroup $\Delta^+(N)$ as defined in $\S$\ref{notation}. 
Since $H \in \mathrm{Max}$, so is $\Delta^+(N)$. Thus, as $\Delta^+(N)$ locally finite, it is actually finite.
 Moreover, since $\Delta^+(N)$ is characteristic in $N$, it is normal in $H$. 
If $\Delta^+(N) \neq 1$ then $H/\Delta^+(N) \in \mathcal{P}$ by \eqref{key}, whence $H \in \mathcal{P}$. This contradicts our starting 
hypothesis. So $\Delta^+(N) = 1$ and hence $kN$ is prime by Connell's Theorem \cite[Theorem~4.2.10]{P}.

Now let $1 \neq M_1 \triangleleft H$ with $|H : M_1| = \infty$. Since $kH$ is a free $kM_1$-module, $kM_1$ is noetherian with 
$\Kdim kM_1 \leq n$. If $\Kdim kM_1< n$ then $M_1 \in \mathcal{P}$ by our choice of $n$, and therefore so is $H$ by \eqref{key}. As this is a contradiction,
\begin{equation}\label{big} \Kdim kM_1 = n.
\end{equation} 
Suppose that there exists a subgroup $T$ of $M_1$ with $1 \neq T \triangleleft M_1$ and $|M_1 : T| = \infty$. Since $kM_1$ is prime 
by \eqref{prime} and $k(M_1 /T) \cong kM_1/\mathfrak{t}kM_1$, where $\mathfrak{t}$ is the augmentation ideal of $kT$, $\Kdim k(M_1/T) <n$ by \cite[Proposition~6.3.11]{McCR}. Thus $M_1/T \in \mathcal{P}$, again by the choice of $n$. 
Now $M_1/T$ is infinite and polycyclic-by-finite, hence it is poly-(infinite cyclic)-by-finite by \cite[Lemma~10.2.5]{P}. Thus we can choose 
normal subgroups $K$ and $L$ of $M_1$, with 
$$ T \subseteq K \subset L \subseteq M_1, \ \ |M_1/L| < \infty, \ \ L/K \textit{ infinite abelian.}$$
Since $kM_1$ is noetherian $M_1$ is finitely generated. Since there are only finitely many homomorphisms from a fixed finitely generated group
 onto a given finite group, there are only finitely many normal subgroups $L_i$ of $M_1$ with $M_1/L_i \cong M_1/L$. In particular, there are
 only finitely many $H$-conjugates of $L$, all of them being in $M_1$ since $M_1 \triangleleft H$. List these as $L = L_1, \ldots , L_r$ and
 define $ \widehat{L} \; := \; \bigcap_{i = 1}^r L_i. $
Therefore 
$$ K \cap \widehat{L} \; \subset \; \widehat{L} \; \subseteq \; M_1, $$
with $\widehat{L} \triangleleft H$ and $M_1/\widehat{L}$ finite. Furthermore, $ \widehat{L}/K \cap \widehat{L} \; \cong \; K\widehat{L}/K $
is infinite abelian, since $L/K$ is infinite abelian and $L/K\widehat{L}$ is finite. Define $M_2$ to be the derived subgroup $[\widehat{L}, \widehat{L}]$, characteristic in $\widehat{L}$, so that $M_2 \triangleleft H$. Moreover, 
\begin{equation}\label{snagged}H/M_2 \in \mathcal{P}
\end{equation}
since $H/M_1$ and $M_1/M_2$ are both in $\mathcal{P}$. We claim that
\begin{equation}\label{bigger} n \ > \ \Kdim k(H/M_2) \ =\ h(H/M_2) \ >\ h(H/M_1).
\end{equation}
For the first inequality, note that, as $kH$ is prime by \eqref{prime} and $M_2 \neq \{1\}$ by \eqref{snagged}, it follows from
\cite[Proposition~6.3.11(ii)]{McCR} that $\Kdim k(H/M_2)<\Kdim H =n$, as desired. The equality is supplied by \cite[Proposition~6.6.1]{McCR}, noting again \eqref{snagged}. Finally, since $M_1/M_2$ has $\widehat{L}/K\cap\widehat{L}$ as a subfactor, $|M_1 : M_2| = \infty$, and hence
$ h(H/M_2) > h(H/M_1)$. Thus \eqref{bigger} holds.
 
If $\Kdim kM_2 < n$ then our choice of $n$ coupled with \eqref{snagged} yields $H \in \mathcal{P}$, a contradiction. So $\Kdim kM_2= n$. 
Continuing in this way, if $M_2$ is \emph{not} just infinite, then we can proceed as above with $M_2$ in place of $M_1$, and so construct 
 a chain of normal subgroups of $H$,
$$ H \supset M_1 \supset M_2 \supset \cdots \supset M_i \supset \cdots ,$$
with $|M_i : M_{i+1}| = \infty$ and $H/M_i \in \mathcal{P}$ for all $i$. However, for all $i$, \cite[Proposition~6.6.1]{McCR} implies that
$$ \Kdim k(H/M_i) \; = \; h(H/M_i)\geq i-1.$$
 Since $n =\Kdim H \in \mathbb{Z}$ this process must terminate after finitely many steps, say at $M_t$. Then $kM_t$ is prime noetherian of Krull dimension $n$ with $M_t \notin \mathcal{P}$ and $M_t$ just infinite, as required.
\end{proof}
 
\begin{remark}\label{mincrim-rem} If $G$ satisfies the conclusions of Proposition~\ref{mincrim} we call $G$ a \emph{minimal criminal}. 
\end{remark}

\medskip
 
Much can be said about the structure of finitely generated just infinite groups: by a result of Grigorchuk \cite{ GJ}, building on seminal results
 of Wilson \cite{W}, they fall into a trichotomy. The version of this which we state here is quoted from \cite[Theorem~5.6]{BGS}. 
 A \emph{hereditarily just infinite group} is defined in \cite[Definition~5.5]{BGS} to be a residually finite group in which every subgroup
 of finite index is just infinite. We don't give the definition of a \emph{branch group} since we will rule out their occurrence in the
 present context; for that definition, see for example \cite[Definition~1.1]{BGS}. 

\begin{theorem}\label{trichotomy}{\rm (Grigorchuk)} Let $G$ be a finitely generated just infinite
group. Then exactly one of the following holds:
\begin{enumerate}
\item[(i)] $G$ is a branch group.
\item[(ii)] $G$ has a normal subgroup $H$ of finite index of the form
\begin{equation}\label{t-defn} H \; = \; L_1 \times \cdots \times L_t, 
\end{equation}
where the factors $L_i$ are copies of a group $L$, conjugation by $G$ transitively
permutes the factors $L_i$, and $L$ has exactly one of the following two properties:
\begin{enumerate}
\item[(a)] $L$ is hereditarily just infinite (in which case $G$ is residually finite);
\item[(b)] $L$ is simple (in which case $G$ is not residually finite).\qed 
\end{enumerate}
\end{enumerate}
\end{theorem}

We review these three possibilities for a minimal counterexample to Conjecture~\ref{Krullqn}. Regarding the first of them Bartholdi, Grigorchuk and Sunik \cite[Theorem~5.7]{BGS} record the following result of Wilson \cite{W}. For this, define an equivalence
relation on the set of subnormal subgroups of a group $G$ by setting $H\equiv K$ if $H \cap K$
has finite index both in $H$ and in $K$. The set $\mathcal{L}(G)$ of equivalence classes of subnormal subgroups,
ordered by the order induced by inclusion, forms a Boolean lattice called the \emph{structure lattice} of $G$.

\begin{theorem}\label{branch}{\rm (Wilson)} Let $G$ be a just infinite group. Then $G$ is a branch group if and only if it has
infinite structure lattice. Moreover, in such a case, the structure lattice is isomorphic to the lattice of closed and open
subsets of the Cantor set.\qed
\end{theorem}

It is shown in \cite[p.~386]{W} that $\mathcal{L}(G)$ embeds into the lattice of subnormal subgroups of $G$. Thus, since our minimal 
criminal $G$ identified in Proposition~\ref{mincrim} is just infinite with the maximum condition on subgroups, its structure lattice 
$\mathcal{L}(G)$ satisfies ACC. On the other hand, the lattice of closed and open subsets of the Cantor set $C$ does not satisfy DCC, 
since for each $i\ge 1$, $C\cap [0,1/3^i]$ is both closed and open in $C$. But since the complement of a closed and open set is again closed 
and open, we then see that the poset of closed and open subsets of $C$ cannot satisfy ACC. In particular groups from case $(i)$ of 
Theorem~\ref{trichotomy} are barred from being minimal counterexamples to Conjecture~\ref{Krullqn}.

Turn now to a minimal counterexample $ G$ satisfying $(ii)$ of Theorem~\ref{trichotomy}. So $G$ has the properties listed in Theorem ~\ref{trichotomy}$(ii)$, with $kG$ (right) noetherian and $\Kdim kG = n$. Since $kG$ is a free left $kH$-module, $kH$ is also noetherian with $\Kdim kH \leq n$. If in fact $\Kdim kH < n$, then $H \in \mathcal{P}$ and hence also $G \in \mathcal{P}$, a contradiction. Note also that $\Delta^+(H) = \{1\}$, since $\Delta^+(H) \subseteq \Delta^+ (G)$ by definition and $\Delta^+ (G) = \{1\}$ since $G$ is just infinite. Therefore, by the above and Connell's theorem, \cite[Theorem~4.2.10]{P},
\begin{equation}\label{sort} kH \textit{ is prime noetherian with } \Kdim kH = n.
\end{equation}
Suppose, next, that $t>1$ in \eqref{t-defn}. Then $kL_1$ is isomorphic to a proper factor of $kH$ and so, by 
 \eqref{sort} and \cite[Proposition~6.3.11(ii)]{McCR}, $\Kdim kL_1<n$. By our inductive hypothesis this implies that $L_1\in \mathcal{P}$ and hence that $H\in \mathcal{P}$. Once again this is a contradiction and so $t$ must equal $1$. In other words, $H=L$ itself is 
 a minimal criminal. Thus we can replace $G$ by $L$ and assume that $G$ satisfies one of parts $(ii)(a)$ or $(ii)(b)$ from Theorem~\ref{trichotomy}.

Summing up, a minimal counterexample to Conjecture~\ref{Krullqn} has the following properties.

\begin{theorem}\label{summary} Let $k$ be a fixed field. Let $G$ be a group and $n$ a positive integer such that $kG$ is noetherian
 with $\Kdim kG = n$, and suppose that $G \notin \mathcal{P}$. Assume that if $H$ is any group with $kH$ noetherian and
 $\Kdim kH < n$ then $H \in \mathcal{P}$. Then there exists a subfactor $\widehat{G}$ of $G$ with the following properties.
\begin{enumerate}
\item[(i)] $k\widehat{G}$ is prime noetherian with $\Kdim k\widehat{G} = n$.
\item[(ii)] $\widehat{G}$ is amenable.
\item[(iii)] $\widehat{G} \notin \mathcal{P}$, in fact $\widehat{G}$ is not elementary amenable.
\item[(iv)] $\widehat{G} \in \mathrm{Max}$ and, in particular $\widehat{G}$ is finitely generated.
\item[(v)] $\widehat{G}$ is just infinite, and is either $(a)$ hereditarily just infinite and so residually finite, or $(b)$ simple.
\item[(vi)] There exists a division ring $D$ with centre $k$ such that $\widehat{G} \subset GL_t (D)$ for some $t\geq 1$. 
\item[(vii)] Assume that $\cchar k=0 $ and that $k$ contains primitive roots of unity of all orders. Then there is a bound on the orders of the finite subgroups of $\widehat{G}$. If $\widehat{G}$ is not simple, then $\widehat{G}$ has no infinite torsion subgroups.
\item[(viii)] Assume that $\cchar k=0 $. Let $H$ be a subgroup of $\widehat{G}$ with $|H| = \infty = |\widehat{G} \, : \, H|$. Then $N_{\widehat{G}}(H) / H \in \mathcal{P}$ with $h(N_{\widehat{G}} (H)/H) < n$ and $|\widehat{G} \, : \, N_{\widehat{G}}(H)| = \infty$.
\end{enumerate}
\end{theorem} 

\begin{proof} We take $\widehat{G}$ to be a minimal criminal and, if necessary, replace it by the subgroup $L$ from Theorem~\ref{trichotomy}. 

 $(i), (iii)$ Use Proposition~\ref{mincrim} and Lemma~\ref{elamenMax}.

$(ii)$ This follows from $(b) \Longrightarrow (a)$ of Theorems~\ref{done} and \ref{pdone}.

$(iv)$ Clear.

$(v)$ This follows from the discussion before the theorem.

$(vi)$ By $(i)$ $k\widehat{G}$ is prime noetherian and so, by Goldie's Theorem, $\widehat{G}$ 
 has a simple artinian quotient ring $Q(k\widehat{G})\cong M_t(D)$ for some integer $t$ and division ring $D$. 
 Thus $\widehat{G} \subseteq GL_t(D)$.

Recall the definition of the FC-subgroup $\Delta(\widehat{G})$ from \S\ref{notation}. We claim that 
\begin{equation}\label{unit}\Delta (\widehat{G}) \;=\; \{1\}.
\end{equation} 
If not, then $\Delta(\widehat{G})$ is a non-trivial normal subgroup of $\widehat{G}$ and hence $|\widehat{G}: \Delta(\widehat{G})|<\infty$ since 
$\widehat{G}$ is just infinite. But $\Delta(\widehat{G})$ must satisfy ACC on subgroups, so $\Delta(\widehat{G}) \in \mathcal{P}$ by \cite[Lemma~4.1.5(iii)]{P}; whence $\widehat{G} \in \mathcal{P}$, a contradiction. Thus \eqref{unit} holds. 
It then follows from 
 \cite[Theorem~7.4]{Sm1} that the centre of $Q(k\widehat{G})$, and hence of $D$, equals $k$, as required.

$(vii)$ We can apply $(b)\Longrightarrow (a)$ of Theorem~\ref{done} to conclude that there is a bound on the orders of the finite subgroups of $\widehat{G}$. Suppose that $\widehat{G}$ is not simple, so, by $(v)$, $\widehat{G}$ is residually finite. By Zelmanov's solution of the restricted Burnside problem \cites{Z1, Z2}, a finitely generated, torsion, residually finite group of bounded
 exponent is finite. Since every subgroup of $G$ is finitely generated and residually finite, every torsion subgroup of $\widehat{G}$ is therefore finite.

$(viii)$ As $k\widehat{G}$ is a free $kH$-module, $kH$ is noetherian, and since $\cchar k=0 $, 
 \cite[Theorem~4.2.12]{P} implies that $kH$ is semiprime. 
 Hence, by Goldie's Theorem, $Q(kH)$ exists and is semisimple artinian. 
 In particular, by \cite[Proposition~2.3.5(ii)]{McCR} every essential right ideal of $kH$ contains a regular element. We claim that
\begin{equation}\label{essential} \textit{ the augmentation ideal } A \textit{ of } kH \textit{ is an essential right ideal of }kH.
\end{equation}
Since $kP/A \cong k$, if \eqref{essential} is false then there is a right ideal $I$ of $kH$ with $I \cong k$ as $kH$-modules. In particular, $IA = 0$ and
 $A$ is an annihilator prime ideal of $kH$. But, by \cite[Proposition~2.2.2(ii)]{McCR}, in a semiprime noetherian ring the only annihilator primes are the minimal primes. However the torsion $FC$-subgroup $\Delta^+ (H)$ of $H$ is finite since it is locally finite by 
\cite[Lemma~4.1.5]{P} and $H\in \mathrm{Max}$. Thus, if $B$ denotes the augmentation ideal of 
$k\Delta^+(H)$, then $BkH$ is a prime ideal of $kH$ by Lemma~\ref{FClem2} and \cite[Theorem~4.2.10]{P}. But $BkH \subsetneq A$ 
as $|H| = \infty$, so that $A$ is not a minimal prime. This contradiction proves \eqref{essential}. Hence $A$ contains a regular element $c$ of $kH$. 

For brevity denote $N_{\widehat{G}}(H)$ by $N$. By the freeness of $kN$ as a $kH$-module, $c$ is a regular element of $kN$, while $\Kdim kN \leq \Kdim k\widehat{G}$. Hence, by \cite[Lemma~6.3.9]{McCR}, and as right $kN$-modules,
\begin{equation}\label{less} n \; \geq \Kdim kN \; > \; \Kdim (kN/ckN)\; \geq \; \Kdim (kN/AkN).
\end{equation}
However $kN/AkN \cong k(N/H)$ both as right $kN$-modules and as rings, so that \eqref{less} shows that $\Kdim k(N/H) < n$. Therefore the 
induction hypothesis forces $N/H \in \mathcal{P}$ and then \eqref{less} together with \cite[Proposition~6.6.1]{McCR} show that $h(N/H) < n$.

Suppose finally for a contradiction that $|\widehat{G} : N| = t < \infty$, with (right) transversal $\{g_1, \ldots , g_t\}$. Since $H \triangleleft N$, $H_0 := \bigcap_{i=1}^t H^{g_i} \triangleleft \widehat{G}$, and it is easy to see that $\widehat{G}/H_0 \in \mathcal{P}$ and is infinite. This contradicts the facts that $\widehat{G}$ is just infinite with $\widehat{G} \notin \mathcal{P}$. 
 \end{proof}

\medskip

\begin{remarks}\label{final} $(i)$ The following observation connects back to $\S$\ref{group}. Suppose that $\cchar k = 0$. If the group 
 $\widehat{G}$ from Theorem~\ref{summary} is in class $(v)(a)$ of that result, then $Z(Q(k\widehat{G})) = k$ but the intersection of the 
 coartinian maximal ideals of $k\widehat{G}$ is 0. So $\{0\}$ is a rational ideal of $k\widehat{G}$ which is not locally closed, 
and the Dixmier-Moeglin equivalence fails for $k\widehat{G}$.

$(ii)$ With regard to Theorem~\ref{summary}$(vi)$ we note that $\widehat{G}$ cannot be linear over a field. For, if it were, 
then by the Tits alternative \cite{Tits}  it is either elementary amenable and thus in $\mathcal{P}$; or it contains a non-cyclic free subgroup, in which case $\widehat{G} \notin \mathrm{Max}$. Either way, it does not satisfy the hypotheses of the theorem.

$(iii)$ In \cite[Corollary~1.6]{EJ} it is shown that there exist finitely generated hereditarily just infinite torsion groups. Theorem~\ref{summary}$(vii)$ shows that, at least when $k$ has $\cchar 0$ and is big enough for $\widehat{G}$, such a group cannot occur as a subgroup of $\widehat{G}$. Moreover the Tits alternative shows that such a group cannot be linear over a field. (See also \cite[Question~15.18]{Ko}.)

$(iv)$ In \cite{JM} the first examples were presented of infinite finitely generated simple amenable groups. As noted at \cite[Lemmas~4.1, 4.2]{JM}, they contain infinite locally finite subgroups, and so certainly do not satisfy $\mathrm{Max}$. Thus they cannot be used for the group $\widehat{G}$ in Theorem~\ref{summary}.

$(v)$ In the first (1965) issue of the Kourovka Notebook, M. Kargapolov asked (Question~1.31) whether every residually finite group with the maximum condition on subgroups is in $\mathcal{P}$. According to the latest edition of the Notebook \cite{Ko}, this remains an open question. 
\end{remarks}

\medskip As a weaker form of Conjecture~\ref{Krullqn} we conjecture  that, in Theorem~\ref{summary}(v), $\widehat{G}$ will actually be simple. As
 a slight evidence in favour of  this conjecture we end the paper  by proving it in the first nontrivial case.  Observe that if $\Kdim kG=0$, then $kG$ is artinian and so  $|G|<\infty$  by  \cite[Theorem~10.1.1]{P},  whence  $G \in \mathcal{P}$.
Thus the first nontrivial case is when $\Kdim kG=1$.

\begin{corollary}\label{crikey} Suppose  that $\cchar \Bbbk = 0$ and that $\Bbbk$ contains primitive roots of unity of all orders. 
Let $G\not\in \mathcal{P}$ be a group satisfying  the hypotheses of Theorem~\ref{summary}, but  with  $\Kdim kG= 1$. Then the group $\widehat{G}$ constructed in that theorem satisfies the following properties.
\begin{enumerate}
\item[(1)] $\widehat{G}$ is   simple.
\item[(2)]  The only finite dimensional $\Bbbk\widehat{G}$-modules are the finite direct sums of the trivial module.
\item[(3)] Suppose that the field $\Bbbk$ is uncountable. Then the only proper nonzero ideal of $\Bbbk\widehat{G}$ is the augmentation ideal. 
\item[(4)] If $T$ is a subgroup of $\widehat{G}$ of infinite order then $|N_{\widehat{G}}(T)\, : \, T| < \infty$.

\end{enumerate}
\end{corollary}

\begin{proof}  (1)  Assume that this is false. By  Theorem~\ref{summary}(v),   $\widehat{G}$ is  then  hereditarily just infinite and hence also  residually finite.

  Suppose first   that $\widehat{G}$ is torsion. By Theorem~\ref{done}(b$)\Rightarrow(a)$ there is a bound on the orders of finite subgroups of $\widehat{G}$. By Zelmanov's solution of the restricted Burnside problem \cites{Z1,Z2},  $\widehat{G}$  is therefore finite. This  contradicts the fact that  $\widehat{G} \notin \mathcal{P}$. 

Hence there is an element $x\in \widehat{G}$ with $|x|=\infty$. Obviously $(x-1)$ is regular in $\Bbbk\langle x \rangle\cong \Bbbk[x,x^{-1}]$ and since $\Bbbk\widehat{G}$ is a free $\Bbbk\langle x \rangle$-module, $(x-1)$ remains regular in $\Bbbk\widehat{G}$.  
But  $\Bbbk \widehat{G}$ is prime noetherian and $\mathrm{Kdim}(\Bbbk \widehat{G}) = 1$. Thus if $I := (x - 1)\Bbbk \widehat{G}$  then  $V := \Bbbk \widehat{G}/I$ is a non-zero module of finite composition length. In particular there are only finitely many simple $\Bbbk \widehat{G}$-modules of finite $\Bbbk$-dimension appearing in the composition series of $V$, the sum of whose dimensions, counted with  multiplicities, is therefore finite, say equal to $M$. 

Now choose normal subgroups $\{H_i \, : \, i \in \mathbb{N} \}$ of finite index in $\widehat{G}$ with 
$ \bigcap_{i \in \mathbb{N}} H_i =\{1\}. 
$
For each $i \in \mathbb{N}$, let $\mathfrak{h}_i$ denote the augmentation ideal of $\Bbbk H_i$, and define algebras 
$$ S_i \; := \; \frac{ \Bbbk \langle H_i, x \rangle +  \mathfrak{h}_i \Bbbk \widehat{G} } { \mathfrak{h}_i \Bbbk \widehat{G} } 
\ \ \subseteq \ \  R_i \; := \; \frac{  \Bbbk \widehat{G}}{\mathfrak{h}_i \Bbbk \widehat{G}} \; \cong \; 
\Bbbk (\widehat{G}/H_i).$$
Note that 
$$ S_i  
 \ \cong  \  \Bbbk \langle H_i, x \rangle/(\mathfrak{h}_i \Bbbk \widehat{G} \cap \Bbbk \langle H_i, x \rangle) 
  \ = \; \Bbbk \langle H_i, x \rangle/\mathfrak{h}_i \Bbbk \langle H_i, x \rangle
\  \cong \ \Bbbk (\langle H_i, x \rangle /H_i).
$$
Consequently,   each $S_i$ is (isomorphic to) the group algebra of a finite cyclic subgroup of $\widehat{G}/H_i$ and hence is commutative, while   each $R_i$ is a free right $S_i$-module of finite rank, say $\mathrm{rank}_{S_i}(R_i) := t_i$. 

For each $i$, write  $ \mathfrak{d}_i $ for  the augmentation ideal of $\Bbbk \langle H_i, x \rangle$. 
Since $x - 1 \in \mathfrak{d}_i$ and $|\widehat{G} \, : \, H_i | < \infty$ it follows that
  $\Bbbk \widehat{G}/\mathfrak{d}_i \Bbbk \widehat{G}$ is a finite dimensional factor of $V$. Therefore
$ t_i \; \leq \; M $ for all $ i. $
This implies that each of the algebras $R_i$ satisfies a PI whose degree is bounded above by  $M$. 

Since $ \bigcap_{i \in \mathbb{N}} H_i =\{1\} $ we have 
$\Bbbk \widehat{G} \; \subseteq \; \prod_{i \in \mathbb{N}} R_i,$  so that $\Bbbk \widehat{G}$ also satisfies a polynomial identity. 
  Thus $\widehat{G}$ is  abelian-by-finite by  \cite[Theorem~6.3.8]{P}  
 and hence $\widehat{G}\in \mathcal{P}$, a contradiction. This contradiction implies that $\widehat{G}$ is simple, as claimed.

$(2)$  Amenable groups do not contain free subgroups on 2 generators \cite[p.15-16]{Jus}.  Thus, if $\widehat{H}$ is linear, then the Tits alternative \cite{Tits} implies that $\widehat{G}$ is  a (finitely generated)  solvable-by-finite group and hence  $\widehat{G}\in\mathcal{P}$, a contradiction. Thus $\widehat{G}$  is not $\Bbbk$-linear. Since $\widetilde{G}$ is simple by (1), the only finite dimensional $\Bbbk\widehat{G}$-modules are therefore finite direct sums of the trivial module, as claimed.  

$(3)$ Let $P$ be a non-zero prime ideal of $\Bbbk \widehat{G}$. Since $\mathrm{Kdim}( \Bbbk \widehat{G}) = 1$, $\Bbbk \widehat{G}/P$ is artinian by  \cite[Proposition 3.15]{KL}, and hence is isomorphic to matrices over a division ring $D$ by the Artin-Wedderburn theorem. But since $\Bbbk$ is uncountable and $\widehat{G}$ is finitely generated, $\Bbbk \widehat{G}$ satisfies the Nullstellensatz by \cite[Corollary 9.1.8(i)]{McCR}. This forces $D$ to be finite-dimensional over $\Bbbk$, so $P$ must be the augmentation ideal $A$ of $\Bbbk \widehat{G}$ by part $(2)$. Finally, if $I$ is any proper nonzero ideal of $\Bbbk \widehat{G}$ then, as for $P$, $\Bbbk \widehat{G}/I$ is artinian, so $I$ must contain some power of the unique non-zero prime ideal $A$. But $A$ is idempotent by $(2)$, so $I = A$ as required. 
  
$(4)$ Let $T$ be an infinite subgroup of $\widehat{G}$ and set  $N := N_{\widehat{G}}(T)$. If $T = \widehat{G}$ there is nothing to prove, so assume that $T$ is a proper subgroup. In this case $|\widehat{G} \, : \, T| = \infty$ since otherwise $\widehat{G}$ would not be simple, contradicting (1). It now follows from Theorem~\ref{summary}(viii) that $N/T \in \mathcal{P}$ with 
$ h(N/T) \; < \; \mathrm{Kdim}(\Bbbk \widehat{G}) \; = \; 1.$
So $N/T$ is finite, as claimed. \end{proof}

  \medskip

\section*{Acknowledgement} We thank Nicolas Andruskiewitsch for helpful input.

\end{document}